\documentclass[12pt]{amsart}
\usepackage{amssymb, graphicx}
\usepackage[latin1]{inputenc}
\usepackage{amsfonts}
\usepackage{latexsym}
\usepackage{amscd}
\usepackage{enumerate}
\RequirePackage{color}

\def\End{\mathop{\text{End}}}
\def\Z{\mathbb Z}

\def\Ann{\text{Ann}}
\def\Der{\mathrm{Der \,}}
\def\SDer{\mathrm{SDer \,}}

\def\ad{\mathrm{ad}}
\def\tr{\text{tr}}

\def\lan{\mathrm{lan}}
\def\ran{\mathrm{ran}}
\newcommand{\gl}{\mathfrak{gl}}
\newcommand{\slinear}{\mathfrak{sl}}
\newcommand{\ort}{\mathfrak{o}}
\newcommand{\simplectic}{\mathfrak{sp}}

\newtheorem{prop}{Proposition}[section]
\newtheorem{lem}[prop]{Lemma}
\newtheorem{thm}[prop]{Theorem}
\newtheorem{cor}[prop]{Corollary}

\newtheorem{definitions}[prop]{Definitions}
\newtheorem{ejm}[prop]{Example}

\newtheorem{vacio}[prop]{ }

\newtheorem{rem}[prop]{Remark}

\numberwithin{equation}{section}

\newcommand{\R}{\mathbb{R}}
\newcommand{\C}{\mathbb{C}}
\newcommand{\M}{\mathbb{M}}

\newcommand{\QAn}{\mathrm{QAnn}}
\newcommand{\Qan}{\mathrm{QAnn}_}
\newcommand{\ann}{\mathrm{ann}}

\newcommand{\der}{\mathrm{Der}}
\newcommand{\dergr}{\mathrm{Der}_{gr}}

\newcommand{\Igre}{\mathcal{I}_{gr-e}}

\newcommand{\Ad}{\mathrm{ad}_}

\newcommand{\Hom}{\mathrm{Hom}}

\newcommand{\sk}{\mathrm{skew}}
\newcommand{\sym}{\mathrm{sym}}

\begin{document}

\title{Martindale algebras of quotients of graded algebras}

\thanks{Partially supported by the MINECO through Project MTM2010-15223 and the Junta de Andaluc\'ia and fondos FEDER through projects FQM 336 and FQM 02467, and by a grant by the Junta de Andaluc\'ia (BOJA n. 12 de 17 de enero de 2008).}

\thanks{{\em Math. Subj. Class.} 17B60, 16W25, 16S90, 16W10} 

\author{H. Bierwirth, C. Mart\'in Gonz\'alez}
\address{Departamento de \'Algebra, Geometr\'ia y Topolog\'ia, Universidad de 
Málaga,  29071 Málaga, Spain}
\email{H. Bierwirth, hannes@uma.es. C. Mart\'in Gonzalez, candido@apncs.cie.uma.es. M. Siles Molina, msilesm@uma.es}
\author{J. S\'anchez Ortega}
\address{Department of Mathematics and Statistics.
2025 TEL Building
York University
North York, Ontario M3J 1P3, Canada}
\email{J. S\'anchez Ortega, jsanchez@fields.utoronto.ca.}
\author{M. Siles Molina}


\begin{abstract}

The motivation for this paper has been to study the relation between the zero component of the maximal graded algebra of
quotients and the maximal graded algebra of quotients of the zero component, both in the Lie case and when considering
Martindale algebras of quotients in the associative setting. We apply our results to prove that the finitary complex Lie algebras are (graded) strongly nondegenerate and compute their maximal algebras of quotients.

\end{abstract}

\maketitle


\section*{Introduction}

Algebras of quotients in the associative setting were introduced in order to find overalgebras more tractable and having ``good properties'' inherited by the algebras they contain as algebras of quotients. This is the case, for example, of the classical algebra of quotients of a semiprime noetherian algebra 
and, for a wider class of algebras, of the Martindale algebra of quotients, $Q_s(A)$, of a semiprime associative algebra $A$, whose study has played an increasingly important role in a deeper knowledge of prime and semiprime algebras.

With this spirit in mind  were introduced the notions of algebra of quotients of a Lie algebra (see \cite{msm}) and that of graded algebra of quotients of a Lie algebra (in \cite{ss}). Maximal (and graded maximal) algebras of quotients do exist and are computed for some important (graded) Lie algebras arising from associative algebras in \cite{bisi, bpss}. 

The authors in \cite{bisi, bpss, CabreraSanchez, ps, PS_SND, ss}  also answered to some natural questions concerning these notions such as when the (graded) maximal algebra of quotients of an (graded) essential ideal $I$ of a (graded) Lie algebra $L$ coincides with the (graded) maximal algebra of quotients of $L$, if taking maximals is a closed operation, or if the grading is inherited by maximal algebras of quotients.

There is however a question which still remained opened. {(Q)}: Does the maximal algebra of quotients of the zero component coincide with the zero component of the maximal algebra of quotients when there is a grading on the underlying algebra? And also  the analogous  in the associative setting (unanswered as far as we know), i.e., {(Q')}:  Is the Martindale algebra of quotients of the zero component of a graded associative algebra $A$ the zero component of the Martindale algebra of quotients of $A$?

Both questions are closely related because, as we  show in Section 3 (see Proposition \ref{QAA=QA}), for $A$ a 3-graded simple associative algebra, the maximal algebra of quotients of $A^-/Z_{A}$ coincides with $Q_s(A)^-/Z_{Q_s(A)}$, and we get a similar result when the algebra $A$ has an involution (see Proposition \ref{QKK=QK}), i.e.,  the maximal algebra of quotients of $K_A/Z_{K_A}$ is just $K_{Q_s(A)}/Z_{K_{Q_s(A)}}$ (here, for $B$ an associative algebra, $Z_B$ denotes the center of  $B$ and when $B$  has an involution $K_B$ denotes its Lie algebra of skew-symmetric elements). 

In Section 3 we also give a positive answer to question {(Q)}, while the answer (also affirmative) for  {(Q')} is given in Section 2. The idea to get this result is to show first that for an idempotent $e\in Q_s(A)$ the corner of the Martindale algebra of quotients of $A$ is the Martindale algebra of quotients of the corner $eAe$, whenever $eA+Ae\subseteq A$. In fact, this is a corollary of a more general result (Theorem \ref{isomorfismolocales}) which states that taking local algebras at elements and Martindale symmetric algebras of quotients commute (the notion of local algebra at an element is a generalization of the concept of corner). Then we show that every finite $\mathbb{Z}$-grading in a semiprime associative algebra $A$ (with more extra conditions) is induced by a set of orthogonal idempotents lying in $Q_s(A)$ whose sum is 1, and the proof follows. This is precisely what we do in Section 2.

Section 1 is devoted to  study the relationship of graded strong nondegeneracy of a graded Lie algebra and that of their zero component and of their graded ideals.

Finally, in Section 4 we apply some of the results of the paper to establish that every finitary simple Lie algebra over an algebraically closed field of characteristic $0$ admits a 3-grading and is strongly nondegenerate. A computation of their maximal algebras of quotients is also provided.


\section{Graded ideals in strongly nondegenerate Lie algebras}

Let $\Phi$ be an associative and commutative ring. Recall that a \textit{Lie algebra} over $\Phi$ (or a $\Phi$-\textit{algebra}) is a $\Phi$-module $L$ with a bilinear map $[\ ,\ ]: L \times L \rightarrow L$, denoted by $(x,y) \mapsto [x,y]$ and called the \textit{bracket} of $x$ and $y$ such that $[x,x]=0$ for all $x \in L$ and $[ \ ,\ ]$ satisfies the \emph{Jacobi identity}, i.e., $[x,[y,z]]+[y,[z,x]]+[z,[x,y]]=0$ for every $x, y, z \in L$.
Given a Lie algebra $L$ and an element $x \in L$ we define $\Ad x : L \rightarrow L$ by $\Ad x (y) = [x,y]$. The associative subalgebra of $\End (L)$ generated by $\Ad x$, for all $x \in L$, is denoted by $A(L)$.

Let $S$ be an associative or Lie $\Phi$-algebra and $G$ an abelian group. We will say that $S$ is $G$-\emph{graded} if $S=\oplus_{\sigma \in
G}S_\sigma,$ where $S_\sigma $ is a $\Phi$-submodule of $S$ and $S_{\sigma}S_{\tau} \subseteq S_{\sigma\tau}\mbox{ for every }\sigma,\tau\in G$. Each element $x$
of $S_\sigma$ is called a \textit{homogeneous element}. We will say that $\sigma$
is the \emph{degree of}  $x$ and will write $\deg x = \sigma$. The set
of all homogeneous elements of $S$ will be denoted by $h(S)$. For a subset $X \subseteq S$, $h(X) = h(S) \cap X$ denotes the set of homogeneous elements of $X$.

If $S$ is a $G$-graded algebra, then a subset $X \subseteq S$ is said to be \textit{graded} if whenever $x=\sum_{\sigma\in
G} x_\sigma \in X$ we have $x_\sigma \in X$, for every $\sigma\in G$.

From now on all our algebras will be considered $\Phi$-algebras. Let $T$ be a subalgebra of an algebra $S$. A linear map
$\delta\colon T\to S$ is called a \emph{derivation} if $\delta (xy)=\delta (x)y + x \delta (y)$ for all $x, y\in T$, where juxtaposition denotes the product in $S$. The set of all derivations from $T$ to $S$, which is a $\Phi$-module, is denoted by $\Der (T,\, S)$. By $\der(S)$ we understand the $\Phi$-module $\der(S,\, S)$; it becomes a Lie algebra if we define the
bracket by $[\delta, \mu]=\delta \mu- \mu\delta$ for $\delta, \mu\in
\der(S)$.

A derivation $\sigma \colon S \to S$ is called \emph{inner} if $\sigma = \ad_x$ for some $x \in S$. The Lie algebra of all inner derivations of $S$ will be denoted by $\mathrm{Inn}(S)$.

Suppose now that the algebra $S$ is a $G$-graded algebra, and
that $T$ is a graded subalgebra of $S$. We say that a derivation $\delta \in \Der (T,\, S)$ has \emph{degree} $\sigma \in G$ if it satisfies
$\delta (T_\tau)\subseteq S_{\tau+ \sigma}$ for every $\tau\in G$.
In this case, $\delta$ is called a \emph{graded derivation
of degree} $\sigma$.  Denote by $\dergr (T,\, S)_\sigma$ the set
of all graded derivations of degree $\sigma$. Clearly, it becomes a $\Phi$-module by defining operations in the natural way and, consequently, $\dergr
(T,\,S):=\oplus_{\sigma\in G}\dergr (T,\, S)_\sigma $ is also a
$\Phi$-module.

The Lie algebra of derivations $\Der A$ of an associative algebra $A$ plays an important role throughout the paper. Also the following subalgebra of $\Der A$ is used: if $A$ has an involution $*$ then the set of all derivations of $A$ which commute with the involution is denoted by $\mathrm{SDer}\, A$; it is easy to see that it is a Lie algebra.

Every associative algebra $A$ gives rise to a Lie algebra $A^-$ by defining $[a,b]=ab-ba,$ where juxtaposition denotes the associative product.

A derivation $\delta$ of $A^-$ is called a \emph{Lie derivation}. Note that every derivation of $A$ is also a Lie derivation. The converse is not true in general.

Let  $X$ and $Y$ be subsets of a Lie algebra $L$.
The
\textit{annihilator of} $Y$ \textit{in} $X$ is
defined as the set $\Ann_X(Y):=\{x \in X \ \vert \  [x, Y]=0\}$ while  the \textit{quadratic annihilator of} $Y$
\textit{in} $X$ is the set $\QAn_X(Y):=\{x \in X \ \vert \  [x, [x, Y]]=0\}.$ The \textit{center} of $L$, denoted by $Z(L)$ or $Z_L$, is $\Ann_L(L)$.

We say that a (graded) Lie algebra $L$ is (\emph{graded}) \emph{semiprime}
if for every nonzero (graded) ideal $I$ of $L$, $I^2 = [I, I] \neq 0$. It is easy to see that if $L$ is (graded) semiprime then
$I\cap \Ann (I) = 0$ for every (graded) ideal $I$ of $L$ \cite[Lemma 1.1\,(ii)]{ss}.
Next, $L$
is said to be (\emph{graded}) \emph{prime} if for nonzero (graded) ideals
$I$ and $ J$ of $L$, $[I, J]\neq 0$. A (graded) ideal
$I$ of $L$
is said to be ({\it graded}) {\it essential} if its intersection with
any nonzero (graded) ideal is again a nonzero (graded) ideal. If $L$ is (graded) semiprime, then a (graded) ideal $I$ of
$L$ is (graded) essential if and only
if $\Ann (I) = 0$ \cite[Lemma 1.1\,(iii)]{ss}. It is easy to see that in
this case $I^2$ is also a (graded) essential ideal.
Further, the intersection of (graded) essential ideals is clearly again a (graded) essential
ideal.
Note also that a nonzero (graded) ideal of a (graded) prime algebra is
automatically (graded) essential.

An element $x$ of a Lie algebra $L$ is an \emph{absolute
zero divisor} if $(\Ad x)^2=0$, equivalently, if $x \in \QAn \, L$. The algebra $L$ is said to be (\emph{graded}) \emph{strongly nondegenerate} if it does not contain nonzero
(homogeneous) absolute zero divisors. It is obvious from the
definitions that (graded) strongly nondegenerate Lie algebras are
(graded) semiprime, but the converse does not hold (see \cite[Remark
1.1]{msm}).

The notions of (graded) semiprimeness, (graded) primeness, and essentiality of an (graded) ideal for associative algebras are defined in exactly the same way as for Lie algebras, just by replacing the bracket by the associative product.

\begin{prop}\label{grequiv}
Let $L=\oplus_{\sigma \in G}L_{\sigma}$ be a graded Lie algebra over an ordered group $G$ with a finite grading. Then  $L$ is graded strongly nondegenerate if and only if $L$ is strongly nondegenerate.
\end{prop}
\begin{proof}
Suppose first that $L$ is graded strongly nondegenerate and suppose there exists a nonzero element $x \in \QAn (L)$. Since $G$ is ordered and the grading is finite, in the decomposition of $x = \sum_{\sigma \in G} x_{\sigma}$ there is a component $x_{\tau}$ of maximum degree $x_{\tau} \neq 0$. Let $y$ be a homogeneous element of $L$, then, in the decomposition of $[x,[x,y]]$ into homogeneous components, the element $[x_{\tau},[x_{\tau},y]]$ is the only with degree $2\tau + \deg y$. Since $[x,[x,y]]=0$, it follows that $[x_{\tau},[x_{\tau},y]]=0$. Applying that $L$ is graded nondegenerate we have that $x_{\tau}=0$, a contradiction. We have proved that $\QAn (L) = 0$ and therefore $L$ is strongly nondegenerate.

The converse is obvious.
\end{proof}

\begin{prop}\label{herencia1}
Let $L=L_{-1}\oplus L_0\oplus L_1$ be a 3-graded Lie algebra with $Z(L)=0$ and $L_0=[L_1, L_{-1}]$, and
suppose that $\Phi$ is 2 and 3-torsion free. If $L_0$ is a strongly nondegenerate Lie algebra, then $L$ is a (graded) strongly nondegenerate Lie algebra.
\end{prop}

\begin{proof}
In view of Proposition \ref{grequiv} it is sufficient to prove that $L$ is graded strongly nondegenerate. We have to show that $(\Ad {x_i})^2=0$ for $x_i\in L_i$ implies $x_i=0$ ($i= 0, \pm 1$). For $i=0$ there is
nothing to prove. Consider $i=\pm 1$ and $x_i\in L_i$ such that $(\Ad {x_i})^2=0$, that is, $x_i \in
\QAn_{L}(L)$. Observe that being $Z(L)=0$ implies that the map $u\mapsto \Ad u$  gives a monomorphism from
$L$ into $A(L)$; this
allows to apply \cite[Theorem 2.1]{PS_SND} and obtain, for any $a_{-i}\in L_{-i}$, $[[a_{-i},\, x_i],\
[[a_{-i},\, x_i],\, L ]]\subseteq \QAn_{L}(L)$.
 In particular,  $[[a_{-i},\, x_i],\ [[a_{-i},\, x_i],\, L_0
]]\subseteq \QAn_{L_0}(L_0)=0$ (as $L_0$ is strongly nondegenerate). This means $[a_{-i},\, x_i]\in
\QAn_{L_0}(L_0)=0$, and we have just proved $[L_{-i},\, x_i]=0$.

Now we claim that $[L_0, x_i]=0$. To this end, and since $L_0=[L_i, L_{-i}]$, it is enough to check $[[b_i,
b_{-i}], x_i]=0$, for every $b_i\in L_i$, with $i=\pm 1$. Indeed, by Jacobi's identity, $[[b_i, b_{-i}],
x_i]= [[b_i, x_i], b_{-i}]+ [[x_i, b_{-i}],\, b_i ]=0$. Hence we have $[L, x_i]=0$, which implies $x_i=0$
because $Z(L)=0$.
\end{proof}

\begin{ejm}\label{recipfalso}
{\rm  The converse of Proposition \ref{herencia1} is not true: Consider  $L:=\mathfrak{sl}_2(\R)$, the Lie
algebra consisting of matrices in $\mathbb{M}_2(\mathbb{R})$ with zero trace, and the grading:
$$L_{-1}=\mathbb{R}e_{21}, \,L_1= \mathbb{R}e_{12}\, \hbox{ and } \, L_0= \mathbb{R}(e_{11}- e_{22}),$$ where
$e_{ij}$ denotes the matrix having 1 in the $(i, j)$-component and 0 otherwise. Then $L$ is a semisimple Lie
algebra with $Z(L)=0$ and $L_0=[L_1, L_{-1}]$. However, while $L$ is strongly nondegenerate, $L_0$ is not, since $L_0$ is an abelian Lie algebra.}
\end{ejm}

\begin{prop}\label{homcuad}
Let $I$ be a graded ideal of a graded Lie algebra $L$ such that $Z(L)=0$. If $x\in h(\QAn_I(I))$ then $[x,\, [x,\, a]]\in h(\QAn_I(I))$ for every $a\in h(L)$.
\end{prop}
\begin{proof}
The condition $Z(L) = 0$ implies that the map $L \rightarrow A(L)$ given by $x \mapsto \Ad x$ is a monomorphism of Lie algebras, so we may compute in $A(L)$, which has an associative structure, and then translate the obtained identities to $L$.

To ease the notation, we shall temporarily get rid of the prefix $\ad$ and use capital letters $X, A,$ etc. instead of $\Ad x, \Ad a,$ etc.

An equation involving commutators on $L$ is then translated into the corresponding equation with capital letters and commutators in $A(L)$.

Let $x \in h(\Qan {I} (I))$. Then
\begin{equation}\label{eq1}
[x,[x,y]] = 0 \text{ for all } y \in I, \text{ which implies }
\end{equation}
\begin{equation} \label{eq2}
X^2 = 0 \text{ on } I.
\end{equation}
Take $a \in h(L)$, and define $u := [x,[x,a]]$. Then
\begin{equation*}
U = [X,[X,A]] = X^2A-2XAX+AX^2,
\end{equation*}
hence,
\begin{align*}
U^2 &= (X^2A-2XAX+AX^2)(X^2A-2XAX+AX^2) \\
&= X^2AU-2XAX^3A+4XAX^2AX-2XAXAX^2+AX^2U.
\end{align*}
Using \eqref{eq2}, $\rm{Im}(U) \subseteq I$ and that $A$ maps $I$ into $I$ we have
\begin{equation} \label{eq3}
U^2=-2XAXAX^2.
\end{equation}
Now we return to \eqref{eq1} to obtain $X^2Y-2XYX+YX^2 = 0$ for all $Y = \Ad y$ with $y \in h(I)$. Apply again \eqref{eq2} and $\rm{Im}(Y) \subseteq I$ to have
\begin{equation*}
YX^2=2XYX \text{ for all } Y = \Ad y \text{ with } y \in h(I). 
\end{equation*}
Substitute in the last equation $Y$ by $[Y,A]$, which is in $\ad (I)$. Then $(YA-AY)X^2 = 2X(YA-AY)X$, hence $YAX^2 - AYX^2 = 2XYA-2XAYX$. Now, change $Y$ to $X$: $XAX^2 - AX^3 = 2X^2A-2XAX^2$, that is, $XAX^2 = 2X^2A+AX^3$. By \eqref{eq2}
\begin{equation*}
XAX^2=0 \text{ on } I.
\end{equation*}
Hence, by \eqref{eq3}, $U^2=0$ on $I$.
\end{proof}

As a consequence of Proposition \ref{homcuad} we obtain the following result, whose analogous in the non-graded case is \cite[Lemma 10]{Z84}.

\begin{cor}\label{gradid} Every graded ideal of a graded strongly nondegenerate Lie algebra is graded strongly nondegenerate (as a Lie algebra).
\end{cor}

\section{The Martindale symmetric ring of quotients of a graded algebra}

Various constructions of algebras of quotients of associative
algebras are known. If $A$ is any semiprime algebra, we denote by $Q^l_{max}(A)$ the maximal left quotient algebra of $A$ and by $Q_s(A)$ the Martindale symmetric algebra of quotients of $A$. Recall that $Q_s(A)$ can be characterized as those elements $q \in Q^l_{max}(A)$ for which there is an essential ideal $I$ of $A$ satisfying that $Iq+qI\subseteq A$. We refer the reader to \cite{BMM} and to \cite{lam} for an account  on these concepts.

The aim of this section is to show that the zero component of the Martindale symmetric algebra of quotients of a finite $\Z$-graded associative algebra $A$ is the Martindale symmetric algebra of quotients of the zero component $A_0$.
To this end we will see that the grading comes from a family of orthogonal idempotents.

Recall that an associative algebra $A$ is \emph{left} (\emph{right}) \emph{faithful} if $xA = 0$ $(Ax=0)$ implies that $x=0$.

A set of submodules $\{A_{ij} \ \vert \ 0 \leq i,j \leq n \}$ of an algebra $A$ is said to be a \emph{Peirce system} if $A = \sum_{i,j=0}^n A_{ij}$ and $A_{ij}A_{kl} \subseteq A_{il}$ if $j=k$ and $A_{ij}A_{kl} = 0$ if $j \neq k$. 

Now, we will show that finite $\Z$-gradings of an associative algebra $A$ are determined by idempotents $e \in Q_s(A)$.

\begin{vacio} \label{Peircesystem}
{\rm Given a graded associative algebra $A=\oplus^n_{k=-n} A_k$, we define: \begin{align*} H_i & :=A_iA_{-n}A_{n-i}  \, \hbox{ for } \, i\in \{0,\ldots, n\} \\ H_{ij}& :=H_iAH_j \, \hbox{ for } \, i,\, j\in \{0,\ldots, n\} \end{align*} It was proved in \cite[Proposition 2.3 (i)]{msm2} that the set $\{H_{ij} \ \vert \  i,\, j= 0,\ldots, n\}$ is a Peirce system and $A = \oplus H_{i,j=0}^n$. We say that the set $\{H_{ij} \ \vert \  i,\, j= 0,\ldots, n\}$ is a \emph{strict Peirce system}}.
\end{vacio}

Let $S$ be a unital algebra. A family $\{e_1, \dots, e_n\}$ of
orthogonal idempotents in $S$ is said to be {\it complete} if
$\sum_{i=1}^ne_i=1\in S$. Suppose $A=\bigoplus_{k=-n}^nA_k$ to be a
graded algebra, subalgebra of $S$.

\begin{definitions}{\rm  We will say that the
$\Z$-grading of $A$ is {\it induced by} the complete system $\{e_1,
\dots, e_n\}$ of orthogonal idempotents of $S$ if
$H_{ij}=e_iAe_j$, (see \ref{Peircesystem} for the definition of the
$H_{ij}$'s).

In particular, for $n=1$ the grading is induced by a complete
orthogonal system of idempotents $\{e, 1-e\}$ if
$$\left\lbrace\begin{matrix}
A_{-1} & = &(1-e)Ae \cr A_0 &= &eAe\oplus
(1-e)A(1-e) \cr A_1 & =&eA(1-e). \end{matrix}\right. $$ 

In this case we will simply say that the 3-grading is {\it induced by the
idempotent} $e$.
}
\end{definitions}

In the following proposition we can see that the grading of the algebra $A$ is induced by a set of idempotents of $Q_s(A)$.

\begin{prop}\label{idempotentesQs}
Let $A = \oplus_{k=-n}^n A_k$ be a graded right faithful algebra such that $A = id(A_{-n})$, the ideal of $A$ generated by $A_{-n}$, and $A=A_0AA_0$. Then there exists a complete system of orthogonal idempotents $\{e_0, \ldots, e_n\} \subseteq Q_s(A)$ inducing the grading of $A$ and hence inducing a grading on $Q_s(A)$ in such a way that $A_k \subseteq Q_s(A)_k$ for $k = 0, \ldots, n$.
\end{prop}

\begin{proof}
By \cite[Theorem 2.8]{msm2} there exists a complete system of orthogonal idempotents $\{e_0, \ldots, e_n\} \subseteq Q_{max}^l(A)$ such that the grading of $A$ is induced by this set. We claim that $e_i \in Q_s(A)$ for $i = 0, \ldots, n.$ Let us first recall the construction of that set; denote by $\pi_{ij} \colon A \to H_{ij}$ the projection onto $H_{ij}$, where $\{H_{ij} \ \vert \  i,j = 0, \ldots, n\}$ is the strict Peirce system defined in \ref{Peircesystem}. For $k \in \{0, \ldots, n\}$, consider
\begin{equation*}
f_k = \sum_{i=0}^n \pi_{ik} \colon A \to A_k = \sum_{i=0}^n H_{ik}.
\end{equation*}
It was shown in \cite[Theorem 2.8]{msm2} that $f_k \in \End_A(A)$ and for $e_k := [A,f_k] \in Q_{max}^l(A),$ the set $\{e_0, \ldots, e_n\}$ is a complete system of orthogonal idempotents. In order to reach our claim we use the characterization of $Q_s(A)$ in terms of the elements of $Q_{max}^l(A)$ by identifying any $x \in A$ with the class $[A,\rho_x],$ where $\rho_x$ maps $a \in A$ to $ax \in A$. We claim that $Ae_k + e_kA \subseteq A$ for $k \in \{0, \ldots, n\}.$ In fact, given $x = \sum_{i,j=0}^n x_{ij}, y = \sum_{i,j=0}^n y_{ij}$ and $a=\sum_{i,j=0}^n a_{ij},$ we have
\begin{align*}
(\rho_xe_k+e_k\rho_y)(a) &= \rho_x\left(\sum_{i=0}^n a_{ik}\right) + e_k(ay) = \left(\sum_{i=0}^n a_{ik}\right)x+a\left(\sum_{i=0}^n y_{ik}\right) \\
&= \sum_{i,j=0}^n a_{ik}x_{kj} + \sum_{i,j=0}^n a_{ji}y_{ik} = a\left(\sum_{j=0}^n x_{kj} + \sum_{i=0}^n y_{ik}\right) = \rho_{z_k}(a)
\end{align*}
for $z_k = \sum_{j=0}^n x_{kj} + \sum_{i=0}^n y_{ik},$ which concludes the proof.
\end{proof}

Let $X$ be a subset of an associative algebra $A$. The \emph{left} and \emph{right annihilator} of $X$ in $A$ are defined as follows:

$$\lan_A(X) = \{a \in A \ \vert \  ax = 0 \text{ for all } x \in X \},$$ 

$$\ran_A(X) = \{a \in A \ \vert \  xa = 0 \text{ for all } x \in X \},$$

while the \emph{annihilator} of $X$ in $A$ is defined as $$\mathrm{ann}_A(X) = \lan_A(X) \cap \ran_A(X).$$

The notion of local algebra at an element that we are going to recall generalizes that of corner in an algebra and has shown to be very useful when the algebra has not a unit element.

If $a \in A$, then the \emph{local algebra} of $A$ at $a$, denoted by $A_a$, is defined as the algebra obtained from the $\Phi$-module $(A_a,+)$ by considering the product given by $axa \cdot aya = axaya$. In particular, when $e = e^2$ is an idempotent of $A$, the local algebra of $A$ at $e$ is just $eAe$.

Given essential ideals $I$ and $J$ of $A$ we will say that a left $A$-homomorphism $f \colon I \to A$ and a right $A$-homomorphism $g \colon J \to A$ are \emph{compatible} if for every $x \in I$ and $y \in J$, $(x)fy = xg(y)$.

Note that in order to distinguish between left and right $A$-homomorphisms we have used the notation $( \ )f$ for a left $A$-homomorphism and $g( \ )$ for a right $A$-homomorphism.

The following two results extend Lemma 3 and Theorem 3 of \cite{gs4}, respectively. The proofs are parallel to those in the cited paper.

\begin{lem}\label{lema32agsadaptado} 
Let $A$ be a semiprime associative algebra, and let $x$ be a nonzero element in $Q_s(A)$ such that $xA + Ax \subseteq A$. Then for every essential ideal $xIx$ of $A_x$, we have:

\rm{(i)} $AxIxA \oplus \mathrm{ann}_A (AxIxA)$ is an essential ideal of $A$.

\rm{(ii)} $\mathrm{ann}_A (AxIxA) \subseteq \lan_A(x) \cap \ran_A(x).$
\end{lem}

\begin{proof}
(i) For every ideal $J$ in the semiprime algebra $A$, $J \cap \mathrm{ann} J = 0$ and therefore $J \oplus \mathrm{ann} J$ is an essential ideal of $A$.

(ii) Take $a \in \mathrm{ann}_A (AxIxA)$. Suppose that $a \notin \lan_A(x)$, then $ax \neq 0$. Since $xA + Ax \subseteq A$ and by semiprimeness of $A$ there exist $u, v \in A$ such that $0 \neq xuaxvx.$ Apply twice that $xIx$ is an essential ideal of $A_x$ to find $xyx, xzx \in xIx$ such that $0 \neq xuaxvx \cdot xyx \cdot xzx = xuaxvxyxzx \in xuaAxIxA =0$, a contradiction. Analogously we can prove that $a \in \ran_A(x).$
\end{proof}

Recall that an element $x$ in an associative algebra $A$ is \emph{von Neumann regular} if $x = xyx$ for some $y \in A$.

\begin{thm} \label{isomorfismolocales}
Let $A$ be a semiprime associative algebra and let $x \in Q_s(A)$ such that $xA+Ax \subseteq A$. Then $x$ is von Neumann regular if and only if $Q_s(A_x) \cong Q_s(A)_x$ under an isomorphism which is the identity on $A_x$.

\end{thm}

\begin{proof}
Note that $A_x \subseteq Q_s(A)_x$ since $A \subseteq Q_s(A)$. Now, for $Q_s(A)_x$, we will prove the three conditions in \cite[(14.25)]{lam}.

(1). Let $xqx$ be an element of $Q_s(A)_x$, where $q \in Q_s(A)$. Let $I$ and $J$ be essential ideals of $A$ satisfying $qxI, Jxq \subseteq A.$ Then $xqxIx, xJxqx \subseteq A_x$, where $xIx$ and $xJx$ are essential ideals of $A_x$: this can be proved as in \cite[Lemma 2]{gs4}.

(2). Take $xqx \in Q_s(A)_x$ and let $xIx, xJx$ be essential ideals of $A_x$. Suppose that $xJxqx =0$. Recall that $Q_s(A)$ is a left quotient algebra of $A$, so, it can be proved as in by \cite[Proposition 3.2 (v)]{gsFG} that $Q_s(A)_x$ is a left quotient ring of $A_x$. Since $xJx$ is a dense left ideal of $A_x$, $xqx = 0$. Analogously we can prove that $xqx = 0$ if and only if $xqxIx = 0$.

(3). Let $xIx$ and $xJx$ be essential ideals of $A_x$, and let $f \in \Hom_{A_x}(_{A_x}xIx , \ _{A_x}A_x)$ and $g \in \Hom_{A_x}(xJx_{A_x},(A_x)_{A_x})$ be compatible. By Lemma \ref{lema32agsadaptado}, $AxIxA \oplus \mathrm{ann}_A(AxIxA)$ is an essential ideal of $A$. Define

\begin{align*}
\overline{f} \colon & AxIxA \oplus \mathrm{ann}_A(AxIxA) \to A \\
& \sum a_ixy_ixb_i + t \mapsto \sum a_i(xy_ixb_ix)f
\end{align*}

Let us show that $\overline{f}$ is well defined. If $\sum a_i(xy_ixb_ix)f \neq 0$ then, by semiprimeness of $A$ there is an element $t \in A$ such that:
\begin{align*}
0 & \neq xt\sum a_i(xy_ixb_ix)f = \sum xta_i(xy_ixb_ix)f = \sum xta_ix \cdot (xy_ixb_ix)f \\
& = \left(\sum xta_ixy_ixb_ix\right)f.
\end{align*}
It follows that $\sum a_ixy_ixb_i \neq 0.$ Moreover, $\overline{f}$ is easily seen to be a homomorphism of left $A$-modules. Analogously we have that the map
\begin{align*}
\overline{g} \colon & AxJxA \oplus \ann_A(AxJxA) \to A \\
& \sum a_ixy_ixb_i + t \mapsto \sum g(xa_ixz_ix)b_i
\end{align*}
is well defined and a homomorphism of right $A$-modules. The homomorphisms $\overline{f}$ and $\overline{g}$ are compatible. Since $x$ is von Neumann regular we can choose $u \in Q_s(A)$ such that $x = xux$ and $u = uxu$. For every $axyxb + t \in AxIxA \oplus \ann (AxIxA)$ and every $a'xy'xb' + t' \in AxJxA \oplus \ann (AxJxA)$ we have
\begin{align}\label{Qsec1}
(axyxb+t)\overline{f}(a'xy'xb'+t') & = a(xyxbx)f(a'xy'xb'+t') \nonumber \\ &= a(xyxbx)fux(a'xy'xb'+t').
\end{align}
From Lemma \ref{lema32agsadaptado} (ii) we obtain
\begin{align} 
 a(xyxbx)fux(a'xy'xb'+t') &= a(xyxbx)f(uxa'xy'xb') \label{Qsec2} 
\end{align}
Since $f$ and $g$ are compatible, $a(xyxbx)f(uxa'xy'xb') = axyxbxug(xa'xy'x)b'$. The last equation together with \eqref{Qsec1} and \eqref{Qsec2} implies that $(axyxb+t)\overline{f}(a'xy'xb'+t') = axyxbxug(xa'xy'x)b'$. Now, apply again Lemma \ref{lema32agsadaptado} to get
\begin{align*}
(axyxb+t)\overline{f}(a'xy'xb'+t') & = axyxbxug(xa'xy'x)b' = (axyxbxu+t)g(xa'xy'x)b' \\ &= (axyxbxu+t)\overline{g}(xa'xy'xb'+t').
\end{align*}
By Lemma \ref{lema32agsadaptado} (i), $AxIxA \oplus \ann (AxIxA)$ and $AxJxA \oplus \ann (AxJxA)$ are essential ideals of $A$. Then, the compatibility of $\overline{f}$ and $\overline{g}$ imply that there exists $q \in Q_s(A)$ such that $(axyxb+t)f=(axyxb+t)q$ and $g(a'xy'xb'+t') = q(a'xy'xb'+t').$ We claim that $q = xuqux.$ Indeed $(axyxb+t)(qux-q) = (axyxbx)fux-(axyxbx)f=0$ for every $axyxb+t \in AxIxA \oplus \ann(AxIxA),$ which is a dense left ideal of $A$. This implies $qux-q=0$, and analogously we obtain $xuq=q$. Finally, it is easy to see that for every $xaxyxbx \in xAxIxAx$ and every $xa'xy'xb'x \in xAxJxAx, (xaxyxbx)f = (xaxyxbx) \cdot q$ and $g(xa'xy'xb'x) = q \cdot (xa'xy'ex'e),$ and hence $(xyx)f = (xyx) \cdot q$ and $g(xy'x) = q \cdot (xy'x)$ for every $xyx \in xIx$ and $xy'x \in xJx$, because $xAxIxAx$ and $xAxJxAx$ are dense left ideals of $xIx$ and $xJx$, respectively, and two left (right) $A_x$-homomorphisms which coincide on a dense left (right) ideal of $xIx$ (resp. $xJx$) coincide on $xIx$ (resp. $xJx$). 

Finally, if $Q_s(A_x) \cong Q_s(A)_x$ then $Q_s(A)_x$ is a unital ring with $x$ as the unit element and therefore $x$ is von Neumann regular.
\end{proof}

Every idempotent is von Neumann regular. Then, a direct consequence of Theorem \ref{isomorfismolocales} is the following:

\begin{cor} \label{Q(eAe)=eQe}
Let $A$ be a semiprime associative algebra. Then for every idempotent $e \in Q_s(A)$ such that $eA+Ae \subseteq A$ we have $Q_s(eAe) \cong eQ_s(A)e$.
\end{cor}

\begin{thm}\label{associativeisomorphism}
Let $A = \oplus_{k=-n}^n A_k$ be a graded right faithful algebra such that $A = id(A_{-n})$, and $A=A_0AA_0$. 
Then $Q_s(A)_0 \cong Q_s(A_0)$, where $Q_s(A)$ is the Martindale symmetric ring of quotients of $A$ and the grading of $Q_s(A)$ is induced by the grading of $A$.
\end{thm}

\begin{proof}
By Proposition \ref{idempotentesQs} the gradings in $A$ and in $Q_s(A)$ are induced by a complete system of orthogonal idempotents $\{e_0, e_1, \dots, e_n\}$ which are in $Q_s(A)$ and such that $Ae_i+e_iA \subseteq A$ for every $i$. This means, in particular,
$$A_0=\oplus_{i=0}^n e_iAe_i\quad  \text{and also}\quad Q_s(A)_0=\oplus_{i=0}^n e_iQ_s(A)e_i.$$
\medskip
By Corollary \ref{Q(eAe)=eQe}

we have  
$$Q_{\rm s}(e_iAe_i)=e_iQ_{\rm s}(A)e_i\quad \text{for every}\quad i\in \{0, 1, \dots, n\}.$$
\medskip
Then 
$$Q_{\rm s}(A)_0=\oplus_{i=0}^n  e_iQ_{\rm s}(A)e_i= \oplus_{i=0}^n  Q_{\rm s}(e_iAe_i)= Q_{\rm s}(\oplus_{i=0}^n e_iAe_i)=Q_{\rm s}(A_0).$$

\end{proof}

%

\section{The maximal algebra of quotients of a graded Lie algebra}\label{gradedsimpleliealgebras}

We first recall the definition of what is the main object of this
section. Let $L$ be a graded subalgebra of a graded Lie algebra $Q$. We say that
$Q$ is a \emph{graded algebra of quotients} of $L$ if for every nonzero
homogeneous element $q\in Q$ there exists a graded ideal $I$ of $L$ with $\Ann_L (I)=0$ and $0\neq [I, q]\subseteq L$ (see \cite{msm,ss}). We are interested in a particular graded algebra of quotients, the so-called maximal graded one. We now have to confine ourselves to the case where $L$ is a graded semiprime algebra. In order to ease the notation, denote by $\Igre (L)$ the set of
all graded essential ideals of a graded Lie algebra $L$.

Consider the set
$$\mathcal{D}_{gr}:=\{(\delta, \, I)\mid I\in \Igre (L),\, \delta
\in \dergr (I,\,L)\}, $$ and define on $\mathcal{D}_{gr}$ the
following relation: $(\delta, \, I)\equiv (\mu,\, J)$ if and only
if there exists $K\in \Igre (L)$ such that $K\subseteq I\cap J$
and $\delta\vert_K ={\mu\vert}_K$. It is easy to see that $\equiv$ is
an equivalence relation.

Denote by $Q_{gr-m}(L)$ the quotient set $\mathcal{D}_{gr}/\equiv$
and by $\delta_I$ the equivalence class of $(\delta, \, I)$ in
$Q_{gr-m}(L)$, for $\delta \in \dergr (I,\,L)$ and $I\in \Igre
(L)$. Then $Q_{gr-m}(L)$, with the following operations: $$
\delta_I + \mu_J = (\delta + \mu)_{I\cap J}$$ $$\alpha (\delta_I)=
(\alpha \delta)_I$$ $$[ \delta_I, \mu_J] = (\delta\mu -
\mu\delta)_{(I\cap J)^2}$$ (for any $\delta_I, \mu_J\in
Q_{gr-m}(L)$ and $\alpha \in \Phi$)
becomes a $G$-graded Lie algebra over $\Phi$, where
$$(Q_{gr-m}(L))_{\sigma} =\{(\delta_\sigma)_I\mid \delta_\sigma \in \dergr
(I,\, L)_\sigma, \, I\in \Igre (L)\}.$$

$Q_{gr-m}(L)$ is called the \emph{maximal graded algebra of quotients} of $L$.

The same definition is valid for the non-graded case,  considering the trivial grading. The \emph{maximal algebra of quotients} of $L$ is denoted by $Q_{m}(L)$. In this case the algebra $L$ must be semiprime and not only graded semiprime.

Recall that a Lie algebra $L$ is said to be (\emph{graded}) \emph{simple} if it has no proper (graded) ideals and it is not abelian.

Throughout this section we will consider that our algebras are $2$-torsion free.

The aim of this section is to prove that for a simple $(2n+1)$-graded Lie algebra $L= \oplus_{k=-n}^n L_k$ we have an isomorphism between the zero component of its maximal algebra of quotients and the maximal algebra of quotients of its zero component, namely $Q_m(L)_0 \cong Q_m(L_0)$.

A direct consequence of the definition of $Q_m(L)$ is the following:

\begin{rem}\label{noideals}
If a Lie algebra $L$ has no proper essential ideals then $Q_m(L) \cong \Der (L)$.
\end{rem}

In the following theorem we prove, for a class of graded Lie algebras, that the maximal algebra of quotients of the zero component is the zero component of the maximal algebra of quotients.

\begin{prop}\label{QAA=QA}

Let $A$ be a semiprime associative algebra. Then $$Q_m([A,A]/Z_{[A,A]}) \cong Q_m(A^-/Z_A).$$ If $A$ is simple, then $$Q_m([A,A]/Z_{[A,A]}) \cong Q_m(A^-/Z_A) \cong Q_s(A)^-/Z_{Q_s(A)} \cong \Der A.$$

\end{prop}

\begin{proof}
From \cite[Corollary 3.6]{bpss} we have that $Q_m([A,A]/Z_{[A,A]}) \cong Q_m(A^-/Z_A)$.

Now, suppose that $A$ is simple and let us show that $Q_s(A)^-/Z_{Q_s(A)} \cong \Der A$. For every $\delta \in \Der A$  there exists an element $q \in Q_s(A)$ such that $\delta = \Ad q$ (see \cite[Proposition 2.5.1]{BMM}). But $\Ad q$, which is an inner derivation of $Q_s(A)$, is not necessarily an inner derivation of $A$, because $q$ may not lie in $A$. The existence of $q$ for every $\delta \in \Der A$ implies that the following map is an epimorphism of Lie algebras.
\begin{align*}
g \colon Q_s(A)^-/Z_{Q_s(A)} &\to \Der A \\
\overline{q} &\mapsto \Ad q
\end{align*}
Note that the map $g$ is well defined since for every $q \in Q_s(A)$ there exists an essential ideal $I$ of $A$ such that $\Ad q (I) \subseteq A$ (see, for example, \cite[Proposition 2.2.3 (ii)]{lam} and since $A$ is simple, necessarily $I = A$ and so $\Ad q (A) \subseteq A$. To conclude that $g$ is an isomorphism let us show that it is injective. Suppose that $g(\overline{q})=0$ for $q \in Q_s(A)$.  This means that $\Ad q (A) = 0$ but then, by \cite[Lemma 1.3 (iii)]{msm}, $q \in Z_{Q_s(A)}$ and therefore $\overline{q}=0$ in $Q_s(A)^-/Z_{Q_s(A)} $.  We have proved that $Q_s(A)^-/Z_{Q_s(A)} \cong \Der A$.

\end{proof}

\begin{thm} \label{lieisomorphism}
Let $A = \oplus_{k = -n}^n A_k$ be a $(2n+1)$-graded simple associative algebra and consider the Lie algebra $L = [A,A]/Z_A = L_{-1} \oplus L_0 \oplus L_1$ where the $(2n+1)$-grading of $L$ is induced by the grading of $A$. Then $Q_m(L)_0 \cong Q_m(L_0)$.
\end{thm}

\begin{proof}
By \cite[Theorem 2.5]{SSfg} the maximal algebra of quotients of $L$ is graded with a grading induced by that of $A$. This gives sense to the last sentence in the statement.

It is shown in Theorem \ref{associativeisomorphism} that:
\begin{equation}\label{assiso}
Q_s(A)_0 = Q_s(A_0),
\end{equation} 
where $Q_s(A)$ is the Martindale symmetric ring of quotients of $A$ and its grading is induced by the grading of $A$. 

By Proposition \ref{QAA=QA} we have $Q_s(A)^-/Z_{Q_s(A)} \cong \Der A$. Moreover, since by Remark \ref{noideals} $Q_m(L) \cong \Der A$, we have:
\begin{equation}\label{maxasslie}
Q_m(L) \cong Q_s(A)^-/Z_{Q_s(A)}
\end{equation}

The same can be applied to the zero component. 

\begin{equation}\label{zeromaxasslie}
Q_m(L_0) \cong Q_s(A_0)^-/Z_{Q_s(A_0)}
\end{equation}

Take together equations \eqref{assiso}, \eqref{maxasslie} and \eqref{zeromaxasslie} to conclude that $Q_m(L)_0 \cong Q_m(L_0)$.

\end{proof}

Recall that for an associative algebra $A$ with involution an $*$ the set of skew-symmetric elements, $K_A:= \mathrm{Skew}(A,*)$, has a structure of Lie algebra.

\begin{prop}\label{QKK=QK}
Let $A$ be a semiprime associative algebra with involution. Then
$$Q_m([K_A,K_A]/Z_{[K_A,K_A]}) \cong Q_m(K_A/Z_{K_A}).$$

If $A$ is simple, then 
$$Q_m([K_A,K_A]/Z_{[K_A,K_A]}) \cong Q_m(K_A/Z_{K_A})\cong K_{Q_s(A)}/Z_{K_{Q_s(A)}}.$$
\end{prop}

\begin{proof}
By \cite[Theorem 6.1]{marmi} $K_A/Z_{K_A}$ is a strongly semiprime Lie algebra. Since every ideal of a prime algebra is essential, $[K_A,K_A]/Z_{[K_A,K_A]}$ is an essential ideal of $K_A/Z_{K_A}$. Then, by \cite[Theorem 3.3]{bpss} $Q_m([K_A,K_A]/Z_{[K_A,K_A]}) \cong Q_m(K_A/Z_{K_A}).$
\end{proof}

\begin{thm}\label{lieisomorphism2}
Let $A$ be a $(2n+1)$-graded simple associative algebra with involution $*$ and consider the Lie algebra $L = [K,K]/Z_{[K,K]}$ where $K = K_A$. Then $Q_m(L)_0 \cong Q_m(L_0)$.
\end{thm}

\begin{proof}
By Proposition \ref{QKK=QK} $Q_m([K,K]/Z) \cong Q_m(K/Z_K)$ and by \cite[Corollary 5.16]{bpss} $Q_m(K/Z_K) \cong \SDer A$. Therefore $Q_m(L) \cong \SDer A$. Similarly as in Theorem \ref{lieisomorphism} we can define a map:
\begin{align*}
g \colon \mathrm{Skew}(Q_s(A)/Z,*) &\to \SDer A \\
\overline{q} &\mapsto \Ad q
\end{align*}
which is an isomorphism of Lie algebras. It follows that
\begin{equation}\label{ecuacion21}
Q_m(L) \cong K_s/Z_{K_s} 
\end{equation}
where $K_s =  \mathrm{Skew}(Q_s(A),\ast)$.

Therefore, we also have
\begin{equation}\label{ecuacion22}
Q_m(L_0) \cong K^0_s/Z_{K^0_s} 
\end{equation}
where $K^0_s = \mathrm{Skew}(Q_s(A_0),\ast).$

From Theorem \ref{associativeisomorphism} we deduce that $(K_s)_0 \cong K^0_s$. Use this together with equations \eqref{ecuacion21} and \eqref{ecuacion22} to conclude that $Q_m(L_0) \cong Q_m(L)_0$.
\end{proof}

Let $L$ be a Lie algebra. A subalgebra $I$ of $L$ is said to be an \emph{inner ideal} if $[I,[I,L]] \subseteq I$. A nonzero inner ideal $I$ is said to be a \emph{minimal inner ideal} if $I$ does not contain strictly nonzero inner ideals of $L$. An \emph{abelian minimal inner ideal} is a minimal inner ideal $I$ such that $[I,I] = 0$.

\begin{thm}\label{casosemisimple}
Let $L$ be a simple $(2n+1)$-graded strongly nondegenerate Lie algebra over a field of characteristic different from $2$ and $3$ containing abelian minimal inner ideals. Then $Q_m(L_0) \cong Q_m(L)_0$.
\end{thm}

\begin{proof}
Since $L$ simple it is isomorphic to one of the Lie algebras $(i), (ii)$ or $(iii)$ of \cite[Theorem 5.1]{dfgg}. If $L$ is as in $(i)$ it is finite dimensional, and then by \cite[Lemma 3.9]{msm} it is trivially true as $Q_m(L) = L$. If $L$ is as in $(ii)$ apply Theorem \ref{lieisomorphism} and if $L$ is as in $(iii)$ apply Theorem \ref{lieisomorphism2}. In each case $Q_m(L_0) \cong Q_m(L)_0$.
\end{proof}


\section{Finitary simple Lie algebras}

In this section we will apply our results to some well known Lie algebras, the so called finitary simple Lie algebras.

Let $V$ be a vector space over a field $F$. An element $x\in \End_F V$ is called {\it finitary} if $\dim xV < \infty$. The set $\mathfrak{fgl}(V)$ consisting of all finitary endomorphisms of $V$ is an ideal of the Lie algebra $\mathfrak{gl}(V):= {\left(\End_F V\right)}^-$. A Lie algebra is called {\it finitary} if it is isomorphic to a subalgebra of $\mathfrak{fgl}(V)$ for some $V$. Finitary Lie algebras were studied by Baranov in \cite{Baranov}, where a classification of infinite dimensional finitary simple Lie algebras over a field of characteristic $0$ is given.

In this section we restrict our attention to a field $F$ of characteristic zero and algebraically closed. From \cite[Corollary 1.2]{Baranov} it  follows that any finitary simple Lie algebra of infinite countable dimension over $F$ is isomorphic to one of the following algebras: 

\begin{enumerate} [{\rm (i)}]
\itemsep=2mm
\item The special Lie algebra $\mathfrak{sl}_\infty(F)$,
\item The orthogonal Lie algebra $\mathfrak{o}_\infty(F)$,
\item The simplectic Lie algebra $\mathfrak{sp}_\infty(F)$.
\end{enumerate}

Where these algebras are defined as follows: let $\M_{\infty}(F)=\cup^\infty_{n=1} \M_n(F)$ be the algebra of infinite matrices with a finite number of nonzero entries, and $\mathfrak{gl}_{\infty}(F)= \M_{\infty}(F)^-$ be its associated Lie algebra. Set
\[y= \rm{diag} \left(
\left(\begin{matrix} 0 & 1 \\ -1 & 0
\end{matrix}\right), 
\left(\begin{matrix} 0 & 1 \\ -1 & 0
\end{matrix}\right),
\dots,
\left(\begin{matrix} 0 & 1 \\ -1 & 0
\end{matrix}\right),
\dots
\right)
\]
Then:
\begin{align*}
\mathfrak{sl}_{\infty}(F) & =\{ x	\in \mathfrak{gl}_{\infty}(F) \ \vert \   \tr(x)=0\}, 
\\
\mathfrak{0}_{\infty}(F) & =\{ x	\in \mathfrak{gl}_{\infty}(F) \ \vert \   x^t + x =0\}, 
\\ 
\mathfrak{sp}_{\infty}(F) & =\{ x \in \mathfrak{gl}_{\infty}(F) \ \vert \   x^t y + yx =0\}. 
\end{align*}

\begin{prop} \label{finitary_grading}
The special, the orthogonal and the simplectic Lie algebras admit a 3-grading $L=L_{-1}\oplus L_0 \oplus L_1$ such that $L_0 = [L_{-1}, L_1]$.
\end{prop}

\begin{proof}
Fix a nonzero natural number $n$. Then one can easily show that the decomposition $\mathfrak{sl}_{\infty}(F)=\mathfrak{sl}_{\infty}(F)_{-1}\oplus \mathfrak{sl}_{\infty}(F)_0 \oplus \mathfrak{sl}_{\infty}(F)_1$, where
\begin{align*}
\mathfrak{sl}_{\infty}(F)_{-1} = &
\left \{
\left (
\begin{array}{cc} 
0 & 0 \\
c & 0
\end{array}
\right) \ \vert \  \, c\in \M_{\infty, \, n}(F) \right \} \cap \mathfrak{sl}_{\infty}(F), 
\\
\mathfrak{sl}_{\infty}(F)_{0} \, \, \, = &
\left \{
\left (
\begin{array}{cc} 
a & 0 \\
0 & d
\end{array}
\right) \ \vert \  \, a\in \M_n(F)
,\, d \in \mathfrak{gl}_{\infty}(F) \, \mbox{ such that } \, \tr (a) + \tr (d) = 0 \right\},
\\
\mathfrak{sl}_{\infty}(F)_{1} \, \, \, = &
\left \{
\left (
\begin{array}{cc} 
0 & b \\
0 & 0
\end{array}
\right) \ \vert \  \, b\in \M_{n,\, \infty}(F) \right \} \cap \mathfrak{sl}_{\infty}(F) 
\end{align*} 
is a 3-grading of $\mathfrak{sl}_{\infty}(F)$ satisfying that $\mathfrak{sl}_{\infty}(F)_0 = [\mathfrak{sl}_{\infty}(F)_{-1}, \mathfrak{sl}_{\infty}(F)_1]$.

Recall that the orthogonal  Lie algebra can be described in an alternative way:
let $V$ be a countable dimensional complex vector space. Set a symmetric bilinear map $f\colon V\times V\to F$, and $\{e_i\}_{i\ge 1}$ an orthonormal basis of $V$ relative to $f$. We have that the orthogonal Lie algebra $\mathfrak{o}_\infty(F)$ is isomorphic to the following one:
\[
\mathfrak{o}(V,f):=\{T\in \mathfrak{fgl}(V) \ \vert \  f(T(x),y)+f(x,T(y))=0, \forall \, x,y\in V\}.
\]
Next, we introduce a new basis of $V$ by doing $e'_i:= \sqrt{-1}e_i$ for all $i$. Then we have $f(e'_i, e'_j) = \delta_{ij}$, where $\delta_{ij}$ denotes Kronecker's delta. Using this basis we can give a matrix representation of $\mathfrak{o}(V,f)$ such that the elements of $\mathfrak{o}(V,f)$ are represented by skew-symmetric matrices with a finite number of nonzero rows and columns. However, in order to describe a $\Z$-grading on $\mathfrak{o}(V,f)$ we shall use a slightly different basis of $V$. Define: 
\begin{align*}
u_1 := \frac{1}{\sqrt 2}&(e'_1 + i e'_2), \qquad \qquad v_1 := \frac{1}{\sqrt 2}(e'_1 - i e'_2) 
\\ 
& \vdots \qquad \qquad \qquad \qquad \qquad \qquad \vdots 
\\
u_n := \frac{1}{\sqrt 2}&(e'_{2n-1} + i e'_{2n}), \qquad v_n := \frac{1}{\sqrt 2}(e'_{2n-1} - i e'_{2n}) 
\\ 
& \vdots \qquad \qquad \qquad \qquad \qquad \qquad \vdots
\end{align*}
Consider the basis $\{u_1, u_2, \cdots; v_1, v_2, \cdots \}$ of $V$. Then $f(u_i, v_i) = 1$ and $f(u_i, u_j) = f(v_i, v_j)= 0$ for $i\neq j$; moreover, $f(u_i, v_j) = 0$ for all $i, j$. Now, given $T\in \mathfrak{o}(V,f)$ writing 
\[ 
T(u_i) = \sum_j \alpha_{ij}u_j + \sum_j \beta_{ij}v_j, \quad T(v_i) = \sum_j \gamma_{ij}u_j + \sum_j \epsilon_{ij}v_j
\]
we obtain $\epsilon_{ij} = - \alpha_{ji}, \, \beta_{ij} = - \beta_{ji}$, and $\gamma_{ij} = - \gamma_{ji}$. Hence, $\mathfrak{o}(V,f)$ admits the following matrix representation 
\[
\left \{ 
\left (
\begin{array}{cc} 
a & b \\
c & -a^t
\end{array}
\right) \ \vert \  \, a\in \mathfrak{gl}_{\infty}(F), \, \mbox{ and } \, b, c \in \sk_{\infty}(F)
\right \},
\]
where $\sk_{\infty}(F) := \{x\in \mathfrak{gl}_{\infty}(F) \ \vert \  x^t = -x\}.$ In the sequel we identify $\mathfrak{o}(V,f)$ with this Lie algebra. Now, it is easy to check that $\mathfrak{o}(V,f) = \mathfrak{o}(V,f)_{-1} \oplus \mathfrak{o}(V,f)_0 \oplus \mathfrak{o}(V,f)_1$ is a 3-grading of $\mathfrak{o}(V,f)$, where 
\[
\mathfrak{o}(V,f)_{-1}   = \,  
\left (
\begin{array}{cc} 
0 & 0 \\
c & 0
\end{array}
\right), \quad 
\mathfrak{o}(V,f)_0  = \, 
\left (
\begin{array}{cc} 
a & 0 \\
0 & -a^t
\end{array} 
\right), \quad
\mathfrak{o}(V,f)_1  = \, 
\left (
\begin{array}{cc} 
0 & b \\
0 & 0
\end{array}
\right),
\]
for $a \in \mathfrak{gl}_{\infty}(F)$ and $b, \, c \in \sk_{\infty}(F)$. Moreover, $\mathfrak{o}(V,f)_0 = [\mathfrak{o}(V,f)_{-1}, \mathfrak{o}(V,f)_1]$.
 \medskip
 
It remains to show that the simplectic Lie algebra $\mathfrak{sp}_{\infty}(F)$ admits a 3-grading. To this end, take an alternate bilinear map $g:V \times V \to F$ such that there exists a basis $\{u_1, u_2, \ldots; v_1, v_2, \ldots \}$ of $V$ satisfying that $g(u_i, v_i) = 1$, $g(u_i, u_j) = g(v_i, v_j)= 0$ for all $i, \, j$, and $g(u_i, v_j) = 0$ for all $i\neq j$. Reasoning as above, we can give a matrix representation of the simplectic Lie algebra 
\[
\mathfrak{sp}_{\infty}(F) \cong \mathfrak{sp}(V,f):=\{T\in \mathfrak{fgl}(V) \ \vert \  g(T(x),y)+g(x,T(y))=0, \forall \, x,y\in V\},
\]
as the Lie algebra 
\[
\left \{ 
\left (
\begin{array}{cc} 
a & b \\
c & -a^t
\end{array}
\right) \ \vert \  \, a\in \mathfrak{gl}_{\infty}(F), \, \mbox{ and } \, b, c \in \sym_{\infty}(F)
\right \},
\]
where $\sym_{\infty}(F) := \{x\in \mathfrak{gl}_{\infty}(F) \ \vert \  x^t = x\}.$ Identifying $\mathfrak{sp}(V,f)$ with this Lie algebra, it is easy to see that $\mathfrak{sp}(V,f) = \mathfrak{sp}(V,f)_{-1} \oplus \mathfrak{sp}(V,f)_0 \oplus \mathfrak{sp}(V,f)_1$ gives a 3-grading of $\mathfrak{sp}(V,f)$ which satisfies $\mathfrak{sp}(V,f)_0 = [\mathfrak{sp}(V,f)_{-1}, \mathfrak{sp}(V,f)_1]$, where 
\[
\mathfrak{sp}(V,f)_{-1}   = \,  
\left (
\begin{array}{cc} 
0 & 0 \\
c & 0
\end{array}
\right), \quad 
\mathfrak{sp}(V,f)_0  = \, 
\left (
\begin{array}{cc} 
a & 0 \\
0 & -a^t
\end{array} 
\right), \quad
\mathfrak{sp}(V,f)_1  = \, 
\left (
\begin{array}{cc} 
0 & b \\
0 & 0
\end{array}
\right),
\]
for $a \in \mathfrak{gl}_{\infty}(F)$ and $b, \, c \in \sym_{\infty}(F)$.
\end{proof}

\begin{thm}
The special, the orthogonal and the simplectic Lie algebras are graded strongly nondegenerate Lie algebras with the grading given in Proposition \ref{finitary_grading}. Moreover, they are strongly nondegenerate Lie algebras.
\end{thm}

\begin{proof}
By Proposition \ref{herencia1}, we only have to show that the zero components of the 3-gradings established in Proposition \ref{finitary_grading} are strongly nondegenerate. We first prove that the Lie algebra $\mathfrak{gl}_{\infty}(F)$ is strongly nondegenerate. 
This is true as follows: it is well-known that $\M_\infty (F)$ is a von Neumann regular associative algebra, and therefore it is semiprime. Since $Z(\M_\infty (F)) = 0$, we can apply \cite[Lemma 4.2]{dfgg} which implies the strong non-degeneracy of $\mathfrak{gl}_{\infty}(F)$. 

Next, we consider the special Lie algebra $L:=\mathfrak{sl}_{\infty}(F)$. We claim that $\QAn_{L_0}(L_0) = 0$; in fact, given
\[
x = \left (
\begin{array}{cc} 
a_0 & 0 \\
0 & d_0
\end{array} 
\right) \in \QAn_{L_0}(L_0), \mbox{ and } y = \left (
\begin{array}{cc} 
a & b \\
c & d
\end{array} 
\right) \in L.
\]
We have
\begin{align*}
[x, y] & = 
\left (
\begin{array}{cc} 
[a_0, a] & a_0b - bd_0 \\
d_0c - ca_0 & [d_0, d]
\end{array} 
\right), \\
[x, [x, y]] & = \left (
\begin{array}{cc} 
[a_0, [a_0, a]] & a^2_0b - 2a_0bd_0 + bd^2_0 \\
d^2_0c - 2d_0ca_0 + cd^2_0 & [d_0,[d_0, d]]
\end{array} 
\right),
\end{align*} 
Then from $[x,[x,y]]= 0$, we obtain $[a_0, [a_0, a]] = [d_0,[d_0, d]] = 0$ for all $a\in \M_n(F)$, $d \in \mathfrak{gl}_{\infty}(F)$. Thus $a_0 = \lambda 1_n$ for some $\lambda \in \C$, and $d_0 = 0$ by the strong non-degeneracy of $\mathfrak{gl}_{\infty}(F)$. This implies that $a^2_0b=0$, and therefore $\lambda^2 b=0$ for all $b\in \M_{n, \infty}(F)$. Hence, $\lambda = 0$, and $x=0$, as desired.

To finish, we deal with the orthogonal and simplectic Lie algebras. Note that the zero component of the 3-gradings stated in Proposition \ref{finitary_grading} have the same zero component: 
\[
\left \{ 
\left (
\begin{array}{cc} 
a & 0 \\
0 & -a^t
\end{array}
\right) \ \vert \  \, a\in \mathfrak{gl}_{\infty}(F)
\right \},
\]
which is a strongly nondegenerate Lie algebra since $\mathfrak{gl}_{\infty}(F)$ is so. The result now follows.
\end{proof}

We finish the section by computing the maximal algebra of quotients of the zero component of the finitary simple Lie algebras described above.

Recall that $\gl_{\infty}(F)_{rcf}$ denotes the algebra of infinite matrices over $F$ with a finite number of entries in each row and each column. Note that $\gl_{\infty}(F) \subsetneq \gl_{\infty}(F)_{rcf}$.

\begin{thm}
Consider the Lie algebras $\mathfrak{sl}_\infty(F), \mathfrak{o}_\infty(F)$ or $\mathfrak{sp}_\infty(F)$ with the $\Z$-grading defined in this section. Consider also the zero component of each of these Lie algebras. The maximal algebra of quotients of each of these algebras is:
\begin{align*} 
&Q_m(\slinear_{\infty}(F)) \cong \gl_{\infty}(F)_{rcf}/F\cdot 1, \\
&Q_m(\ort_{\infty}(F)) \cong \{x \in \gl_{\infty}(F)_{rcf} : x^t+x=0 \},  \\
&Q_m(\simplectic_{\infty}(F)) \cong \{x \in \gl_{\infty}(F)_{rcf} : x^ty+yx=0 \},\\
& Q_m \left(\left \{
\left (
\begin{array}{cc} 
a & 0 \\
0 & d
\end{array}
\right) \ \vert \  \, a\in \M_n(F)
,\, d \in \mathfrak{gl}_{\infty}(F),  \tr (a) + \tr (d) = 0 \right\}\right) \cong (\gl_{\infty}(F)_{rcf}/F\cdot 1)_0, \\
& Q_m \left(\left \{
\left (
\begin{array}{cc} 
a & 0 \\
0 & -a^t
\end{array}
\right) \ \vert \  
\, a \in \mathfrak{gl}_{\infty}(F) \right\}\right) \cong (\{x \in \gl_{\infty}(F)_{rcf} \ \vert \  x^t+x=0 \})_0.
\end{align*}
\end{thm}

\begin{proof}

Let $L$ be one of the Lie algebras $\mathfrak{sl}_\infty(F), \mathfrak{o}_\infty(F)$ or $\mathfrak{sp}_\infty(F)$.
Since $L$ is simple, from Remark \ref{noideals} it follows $Q_m(L) \cong \Der (L)$. Now apply \cite[Theorem I.3]{neeb} in each case to get the desired result.

The other two cases are now a consequence of Theorem \ref{casosemisimple} where we have proved that $Q_m(L)_0 \cong Q_m(L_0)$.

\end{proof}

Note that, since $\ort_{\infty}(F)_0 \cong \simplectic_{\infty}(F)_0$ we have:

$$(\{x \in \gl_{\infty}(F)_{rcf} \ \vert \  x^ty+yx=0 \})_0 \cong (\{x \in \gl_{\infty}(F)_{rcf} \ \vert \  x^t+x=0 \})_0.$$



\begin{thebibliography}{99}



\bibitem{Baranov} {\sc A. A. Baranov}, Finitary simple Lie
algebras, {\it J. Algebra} {\bf 219} (1999), 299--329.


\bibitem{BMM} {\sc K. I. Beidar, W. S. Martindale III, and A. V.
Mikhalev}, {\it Rings with generalized identities}. Marcel Dekker,
New York,  1996.


\bibitem{bisi} \textsc{H.~Bierwirth, M.~Siles Molina},
 Lie ideals of graded associative algebras \emph{Israel J. Math.} DOI: 10.1007/s11856-011-0201-7.

\bibitem{bpss} \textsc{M.~Bre\v{s}ar, F.~Perera, J.~S\'anchez~Ortega, M.~Siles Molina},
 Computing the maximal algebra of quotients of a Lie algebra. \emph{Forum Math.} \textbf{21} (2009), 601--620

\bibitem{CabreraSanchez} \textsc{M.~Cabrera, J.~S\'anchez~Ortega},
 Lie quotients for skew Lie algebras. \emph{Algebra Colloq.} {\bf 16}: 2 (2009), 267-274.


\bibitem{dfgg} \textsc{C.~Draper, A.~Fern\'andez L\'opez, E.~Garc\'ia, M.~G\'omez Lozano},
The socle of a nondegenerate Lie algebra, \emph{J. Algebra} {\bf 319} (2008), 2372-2394.

\bibitem{gsFG} \textsc{M. G\'omez Lozano, M. Siles Molina}, Quotient rings and Fountain-Gould left orders by the local approach, \emph{Acta Math. Hungar.} {\bf 97} (2002), 179--193.

\bibitem{gs4} \textsc{M. A. G\'omez Lozano, M. Siles Molina}, Local rings of rings of quotients, \emph{Algebras and Repr. Theory} {\bf 11 (5)} (2008), 425--436.


\bibitem{lam}\textsc{T.~Y.~Lam},  \emph{Lectures on modules and rings},
Springer-Verlag Berlin Heidelberg New York, 1999.


\bibitem{marmi} \textsc{W. S. Martindale 3rd, C. R. Miers}, Herstein's Lie
theory revisited, \emph{J. Algebra} \textbf{98} (1986), 14--37.



\bibitem{neeb} \textsc{Karl-Hermann Neeb}, Derivations of Locally Simple Lie Algebras, \emph{J. Lie Theory} \textbf{15} (2005), 589---594.



\bibitem{ps} \textsc{F.~Perera, M.~Siles Molina}, Associative and Lie algebras of quotients. \emph{Publ. Mat.} \textbf{52} (2008), 129--149.

\bibitem{PS_SND} \textsc{F.~Perera, M.~Siles Molina}, Strongly nondegenerate Lie
algebras. \emph{Proc. Amer. Math. Soc}. \textbf{136} (2008), no. 12, 4115--4124.

\bibitem{ss} \textsc{J.~S\'anchez~Ortega, M.~Siles Molina}, Algebras of quotients of graded Lie algebras.  \emph{J. Algebra} \textbf{323} (2010), no. 7, 2002--2015.

\bibitem{SSfg} \textsc{J.~S\'anchez~Ortega, M.~Siles Molina}, Finite gradings of Lie algebras.  \emph{J. Algebra} (to appear).

\bibitem{msm} \textsc{M.~Siles Molina}, Algebras of quotients of Lie
algebras, \emph{J. Pure  Appl. Algebra} \textbf{188} (2004),
175--188.

\bibitem{msm2} \textsc{M.~Siles Molina}, Finite $\Z$-gradings of simple associative algebras, \emph{J. Pure  Appl. Algebra} \textbf{207} (2006),
619--630.



\bibitem{Z84} {\sc E. Zelmanov}, Lie algebras with an algebraic adjoint representation, {\it Math.} {\bf USSR-Sb 49} (1984), 537-552.

\end{thebibliography}
\end{document}